\documentclass[11pt,reqno]{amsart}
\usepackage{setspace}
\usepackage{amsmath, amssymb, amscd, amsthm, amsfonts}
\usepackage{graphicx}

\usepackage{mathpazo}

\usepackage{indentfirst}
\usepackage{amssymb,amsmath,amsfonts,amsthm}
\usepackage[all]{xy}
\usepackage{enumerate}
\usepackage{mathrsfs}
\usepackage{graphicx}
 
\numberwithin{equation}{section}
\usepackage[all]{xy}
\usepackage{amssymb,amscd,amsthm,amsmath,graphicx,color}
\usepackage{mathrsfs}
\usepackage[english]{babel}
\usepackage[utf8x]{inputenc}
\usepackage{hyperref}

\newtheorem{theorem}{Theorem}[section]
\newtheorem*{proposition*}{Proposition}
\newtheorem{proposition}[theorem]{Proposition}
\newtheorem{corollary}[theorem]{Corollary}
\newtheorem{lemma}[theorem]{Lemma}

\newtheorem{remark}[theorem]{Remark}
\newtheorem{definition}[theorem]{Definition}
\newtheorem{question}[theorem]{Question}

\DeclareMathOperator{\cidim}{CI-dim}

\DeclareMathOperator{\cd}{cd}
\DeclareMathOperator{\gdim}{G-dim}

\begin{document} 

\title[Ext-Tor duality, cohomological dimension, and applications]{An Ext-Tor duality theorem, cohomological dimension, and applications}

\author{Rafael Holanda}
\address{Departamento de Matemática, Universidade Federal da Paraíba - 58051-900, João Pessoa, PB, Brazil} \email{rf.holanda@gmail.com}

\author{Cleto B. Miranda-Neto}
\address{Departamento de Matemática, Universidade Federal da Paraíba - 58051-900, João Pessoa, PB, Brazil} \email{cleto@mat.ufpb.br}

\date{\today}


\maketitle

\begin{abstract} We provide a duality theorem between Ext and Tor modules over a Cohen-Macaulay local ring possessing a canonical module, and use it to prove some freeness criteria for finite modules. The applications include a characterization of codimension three complete intersection ideals and progress on a long-held multi-conjecture of Vasconcelos. By a similar technique, we furnish another theorem which in addition makes use of the notion of cohomological dimension and is mainly of interest in dimension one; as an application, we show that the celebrated Huneke-Wiegand conjecture in the case of complete intersections holds true provided that a single additional condition is satisfied.

\end{abstract}

\section{Introduction}\label{mcmsection}

As is well-known, duality results are desirable in mathematics, particularly in homological commutative algebra. Our main purpose in this note is to obtain, by means of spectral sequence methods, a new duality theorem involving Ext and Tor modules over a Cohen-Macaulay local ring which admits a canonical module; see Theorem \ref{exttor} (also Remark \ref{differ} for an alternative proof). By means of a similar technique, we also provide Theorem \ref{freeness}.

Then we derive  applications essentially related to the problem of deciding freeness of a given finite module over such a ring, which include a new characterization of ideals that can be generated by a regular sequence of length $3$ (Corollary \ref{appl1}), and progress toward the 1985 multi-conjecture of Vasconcelos \cite{V} concerning fundamental modules such as the normal module, the twisted conormal module and the module of K\"ahler differentials of an algebra which is now assumed to be essentially of finite type over a field containing the rationals (Corollary \ref{appl4}). 

Another main byproduct is concerned with the celebrated Huneke-Wiegand conjecture \cite{HW}, which remains open, for instance, in the case of complete intersection local rings (of codimension at least 2); over such rings, as a corollary of more general cases, we detect one single extra condition related to the concept of (generalized) cohomological dimension under which the conjecture is proven to be true; see Corollary \ref{H-W-version}.

We consider a few conventions that will be in force throughout the paper. By a \textit{ring} we mean a commutative Noetherian ring with multiplicative identity 1. Given a ring $R$, we say that an $R$-module $N$ is \textit{finite} if $N$ is finitely generated. For a local ring $R$ with canonical module $\omega_R$, we set $$N^{\vee}={\rm Hom}_R(N, \omega_R),$$ the canonical dual of $N$. As usual, if ${\bf m}$ is the maximal ideal of $R$, then by ${\rm depth}\,N$ we mean the ${\bf m}$-depth of the $R$-module $N$, and the finite $R$-module $N$ is maximal Cohen-Macaulay if ${\rm depth}\,N={\rm dim}\,R$. By a widely accepted convention, ${\rm depth}\,0=+\infty$, so the zero module is not maximal Cohen-Macaulay. When $R$ is Cohen-Macaulay, a fundamental basic property (which we shall freely use) is that if $N$ is maximal Cohen-Macaulay then so is $N^{\vee}$, and in addition $N\cong N^{\vee\vee}$ (see \cite[Proposition 3.3.10]{BH}). 

\medskip

\noindent{\bf Acknowledgements.} The first-named author was supported by the CNPq-Brazil grant 200863/2022-3. The second-named author was partially supported by the CNPq-Brazil grants 301029/2019-9 and 406377/2021-9.

\section{Main result and first corollaries}

\subsection{Main result}
We establish the following duality theorem between Ext and Tor modules.

\begin{theorem}\label{exttor}
Let $R$ be a Cohen-Macaulay local ring with canonical module $\omega_R$. Let $M$ and $N$ be finite $R$-modules, with $N$ maximal Cohen-Macaulay, and let $n\geq 1$ be an integer. If, for each $j=0, \ldots, n-1$, the $R$-module ${\rm Tor}_j^R(M,N^\vee)$ is either zero or maximal Cohen-Macaulay, then there exist isomorphisms $${\rm Ext}^j_R(M,N)\cong{\rm Tor}^R_j(M,N^\vee)^\vee\quad \mbox{for \,all} \quad j=0, \ldots, n.$$ In particular, ${\rm Ext}^j_R(M,N)^\vee\cong{\rm Tor}^R_j(M,N^\vee)$ for all $j=0, \ldots, n-1$. There is also an exact sequence 
$$\xymatrix@=1em{0\ar[r] & {\rm Ext}^1_R({\rm Tor}_n^R(M,N^\vee),\omega_R)\ar[r] & {\rm Ext}_R^{n+1}(M,N)\ar[r] & {\rm Tor}_{n+1}^R(M,N^\vee)^\vee\ar[dll]\\  & {\rm Ext}^2_R({\rm Tor}_n^R(M,N^\vee),\omega_R)\ar[r] & {\rm Ext}^{n+2}_R(M,N).}$$
\end{theorem}
\begin{proof} We can assume $N\neq 0$. Because $N^{\vee}$ is maximal Cohen-Macaulay and $N\cong N^{\vee\vee}$, there exists a first quadrant spectral sequence
$$E_2^{p,q}={\rm Ext}^p_R({\rm Tor}_q^R(M,N^\vee),\omega_R)\Rightarrow_q{\rm Ext}^{p+q}_R(M,N),$$ which must satisfy $$E_2^{p,q}=0 \quad \mbox{for\, all} \quad p\geq 1 \quad \mbox{and} \quad q=0, \ldots, n-1 .$$ This spectral sequence has therefore the following shape at its second page:
$$\xymatrix@=1em{
0 & 0 & \cdots & 0 & E_2^{n,2} & E_2^{n+1,2} & \cdots
\\
0 & 0 & \cdots & 0 & E_2^{n,1} & E_2^{n+1,1} & \cdots
\\
E_2^{0,0} & E_2^{1,0} & \cdots & E_2^{n-1,0} & E_2^{n,0} & E_2^{n+1,0}\ar[uul] & \cdots
}$$ From this shape and the convergence, we obtain the desired isomorphisms
$${\rm Tor}_j^R(M,N^\vee)^\vee=E_2^{0,j}\cong{\rm Ext}^j_R(M,N) \quad \mbox{for \,all} \quad j=0, \ldots, n$$ as well as the five-term-type exact sequence
$$\xymatrix{0\ar[r] & E_2^{n, 1}\ar[r] & {\rm Ext}_R^{n+1}(M,N)\ar[r] & E_2^{n+1, 0}\ar[r] & E_2^{n, 2}\ar[r] & {\rm Ext}^{n+2}_R(M,N).}$$ 
\end{proof}

\begin{remark}\label{differ}\rm ({\it A different proof.}) Our spectral-sequence-based proof above is short and self-contained. For those readers not familiar with spectral sequence methods, we are also able to provide an alternative proof which, on the other hand, makes use of a strong theorem from \cite{yo}; more precisely, from \cite[Theorem 2.3]{yo} with $X$ therein taken to be ${\bf F}\otimes_RN^{\vee}$ (where ${\bf F}=\{\xymatrix@=1em{\cdots \ar[r] & F_i\ar[r]^{d_i} & F_{i-1}\ar[r] & \cdots\}}$ is a minimal free resolution of our $M$), and $M$ therein taken as our $\omega_R$. Indeed, by the well-known depth lemma, the three short exact sequences \begin{equation}\label{seq1}\xymatrix{0\ar[r] & {\rm Tor}_i^R(M,N^\vee)\ar[r] & {\rm coker}(d_{i+1}\otimes_RN^{\vee})\ar[r] & {\rm im}(d_{i}\otimes_RN^{\vee})\ar[r] & 0}\end{equation} \begin{equation}\label{seq2}\xymatrix{0\ar[r] & {\rm im}(d_{i}\otimes_RN^{\vee})\ar[r] & {\rm ker}(d_{i-1}\otimes_RN^{\vee})\ar[r] & {\rm Tor}_{i-1}^R(M,N^\vee)\ar[r] & 0}\end{equation} \begin{equation}\label{seq3}\xymatrix{0\ar[r] & {\rm ker}(d_{i}\otimes_RN^{\vee})\ar[r] & F_i\otimes_RN^{\vee}\ar[r] & {\rm im}(d_{i}\otimes_RN^{\vee})\ar[r] & 0}\end{equation} enable us to derive that the $R$-modules $${\rm im}(d_{i}\otimes_RN^{\vee}),  \quad {\rm ker}(d_{i}\otimes_RN^{\vee}) \quad \mbox{and} \quad {\rm coker}(d_{i}\otimes_RN^{\vee})$$ are maximal Cohen-Macaulay for all $i=0, \ldots, n$. Moreover, by considering the induced long exact sequence in Ext coming from applying $(-)^{\vee}$ to (\ref{seq1}) with $i=n$, we obtain $${\rm Ext}^j_R({\rm Tor}_n^R(M,N^\vee),\omega_R)\cong {\rm Ext}^j_R({\rm coker}(d_{n+1}\otimes_RN^{\vee}),\omega_R) \quad \mbox{for\, all} \quad j\geq 1.$$ Now,  applying \cite[Theorem 2.3]{yo} the result follows, noting that while $X=\{X^j\}$ is assumed therein to be a complex of projective $R$-modules, this is not needed in the proof as it suffices that ${\rm Ext}_R^i(X^j, M)=0$ for all $i\geq 1$ and all $j$.

\end{remark}

\begin{question}\rm Does Theorem \ref{exttor} remain true if each non-zero  ${\rm Tor}_j^R(M,N^\vee)$, with $0\leq j \leq n-1$, is required to have a sufficiently high depth $($depending on $j$$)$ but possibly smaller than ${\rm dim}\,R$? 
    
\end{question}

\subsection{First corollaries}
The first byproduct is an immediate consequence of Theorem \ref{exttor}. Recall that two finite $R$-modules $M, N$ are said to be Tor-{\it independent} if $${\rm Tor}_j^R(M, N)=0 \quad \mbox{for\, all} \quad j\geq 1.$$

\begin{corollary}\label{lmweak0}
Let $R$ be a Cohen-Macaulay local ring possessing a canonical module, and $M, N$ be finite $R$-modules. Suppose $N$ and $M\otimes_RN^{\vee}$ are maximal Cohen-Macaulay. The following assertions hold:
\begin{itemize}
    \item [(i)] If $n\geq 1$ is an integer and ${\rm Tor}_j^R(M,N^\vee)=0$ for all $j=1, \ldots, n$, then ${\rm Ext}^j_R(M,N)=0$ for all $j=1, \ldots, n$;
\item [(ii)] If $M$ and $N^{\vee}$ are {\rm Tor}-independent, then ${\rm Ext}^j_R(M,N)=0$ for all $j\geq 1$.
\end{itemize}
\end{corollary}

\begin{remark} \rm  It is worth pointing out that, according to \cite[Lemma 3.4]{LM}, one situation where the $R$-module $M\otimes_RN^{\vee}$ is known to be maximal Cohen-Macaulay is when ${\rm Ext}^i_R(M,N)=0$ for all $i=1, \ldots, {\rm dim}\,R$.
\end{remark}

Now let us introduce a piece of standard notation. For a local ring $R$, a finite $R$-module $M$ and an integer $s\geq 1$, we denote by $\Omega_s^RM$ the
$s$-th syzygy module of $M$, namely, the image of the $s$-th differential map in a minimal free resolution
of $M$. We set $\Omega_0^RM=M$. Note that $\Omega_s^RM$ is, up to isomorphism,
uniquely determined since so is a minimal free resolution of $M$.

\begin{question} \rm Is Corollary \ref{lmweak0} still true if we replace the maximal Cohen-Macaulayness of $M\otimes_RN^{\vee}$ with the weaker assumption ${\rm depth}\, \Omega_1^RM\otimes_RN^{\vee}\neq {\rm depth}\,M\otimes_RN^{\vee}+1$? \end{question}

In \cite[Lemma 3.3]{LM} it was proved that, if $R$ is a Cohen-Macaulay local ring with canonical module $\omega_R$, and $M, N$ are finite $R$-modules such that $$N, \quad M\otimes_RN^\vee \quad \mbox{and} \quad \Omega_1^RM\otimes_RN^\vee$$ are maximal Cohen-Macaulay, then ${\rm Ext}^1_R(M,N)=0$ if and only if ${\rm Tor}_1^R(M,N^\vee)=0$. We are able to substantially improve  this result, as a particular case of the following corollary.

\begin{corollary}\label{lmweak}
Let $R$ be a Cohen-Macaulay local ring possessing a canonical module. Let $M$ and $N$ be finite $R$-modules with $N$ maximal Cohen-Macaulay, and $n\geq 1$ an integer. Assume that, for each $j=0, \ldots, n-1$, the $R$-module ${\rm Tor}_j^R(M,N^\vee)$ is either zero or maximal Cohen-Macaulay. The following assertions hold true:
\begin{itemize}
    \item [(i)] Suppose ${\rm Tor}_n^R(M,N^\vee)=0$. Then, ${\rm Ext}^n_R(M,N)=0$ and
$${\rm Ext}^{n+1}_R(M,N)\cong{\rm Tor}^R_{n+1}(M,N^\vee)^\vee.$$ 
    \item [(ii)] Suppose ${\rm Tor}_n^R(M,N^\vee)\neq0$. If both $\Omega_{n-1}^RM\otimes_RN^\vee$ and $\Omega_n^RM\otimes_RN^\vee$ are maximal Cohen-Macaulay, then ${\rm Tor}_n^R(M,N^\vee)$ must be maximal Cohen-Macaulay as well, and
$${\rm Ext}^j_R(M,N)\cong{\rm Tor}_j^R(M,N^\vee)^\vee \quad \mbox{for\, all} \quad j=0, \ldots, n+1.$$ In particular, ${\rm Ext}^n_R(M,N)\neq 0$.
\end{itemize}
\end{corollary}
\begin{proof} Part (i) follows immediately from Theorem \ref{exttor}. To prove (ii),  consider a short exact sequence $$\xymatrix{0\ar[r] & \Omega_n^R M\ar[r] & F\ar[r] & \Omega^R_{n-1}M\ar[r] & 0}$$ with $F$ a free $R$-module. This yields an exact sequence $$\xymatrix{0\ar[r] & {\rm Tor}_1^R(\Omega_{n-1}^RM,N^\vee)\ar[r] & \Omega_n^R M\otimes_RN^\vee\ar[r] & F\otimes_R N^\vee\ar[r] & \Omega_{n-1}^RM\otimes_RN^\vee\ar[r] & 0}$$ 
which can be decomposed into two short exact sequences:
$$\xymatrix{0\ar[r] & {\rm Tor}_1^R(\Omega_{n-1}^RM,N^\vee)\ar[r] & \Omega_n^RM\otimes_RN^\vee\ar[r] & \Omega_n^RM\otimes_RN^\vee/{\rm Tor}_1^R(\Omega_{n-1}^RM,N^\vee)\ar[r] & 0,}$$ 
$$\xymatrix{0\ar[r] & \Omega_n^RM\otimes_RN^\vee/{\rm Tor}_1^R(\Omega_{n-1}^RM,N^\vee)\ar[r] & F\otimes_RN^\vee\ar[r] & \Omega_{n-1}^RM\otimes_RN^\vee\ar[r] & 0.}$$ It follows that the module $${\rm Tor}^R_n(M,N^\vee)\cong{\rm Tor}^R_1(\Omega_{n-1}^RM,N^\vee)$$ is maximal Cohen-Macaulay. Finally, we apply Theorem \ref{exttor}.\end{proof}

\section{Freeness criteria}

In this section, we keep applying spectral sequence methods to derive from Theorem \ref{exttor} some freeness criteria for finite modules over local rings -- a topic which has recently gained much attention, especially over Cohen-Macaulay local rings possessing a canonical module (e.g., complete local rings). We employ separately two basic tools, namely, the Auslander transpose and generalized local cohomology.

\subsection{Auslander transpose}  Recall the following well-known tool, which plays an important role in this part. Let $R$ be a local ring. Given a finite $R$-module $M$, let $$\xymatrix{\cdots\ar[r] & F_2\ar[r]^{d_2} & F_1\ar[r]^{d_1} & F_0\ar[r] & M\ar[r] & 0}$$ be a minimal free resolution of $M$. The {\it Auslander transpose} of $M$ is defined as $${\rm Tr} M={\rm coker}\, {\rm Hom}_R(d_1, R).$$
This module is uniquely determined up to isomorphism.

Our purpose here is to prove a criterion for the freeness of modules making use of the two lemmas below.

\begin{lemma}$($\cite[Proposition 3.2]{KOT}$)$\label{trlemma}
Let $R$ be a local ring and $M, N$ be finite $R$-modules, with ${\rm depth}_RN=0$. If ${\rm Ext}^1_R({\rm Tr} M,{\rm Hom}_R(M,N))=0$, then $M$ is free.
\end{lemma}

We denote by ${\rm pd}_RM$ the projective dimension of an $R$-module $M$, and by $M^*={\rm Hom}_R(M, R)$ its algebraic dual. We believe the spectral sequence given in the lemma below is known to the experts; at any event, we provide a self-contained natural proof. 

\begin{lemma}\label{pdss}
Let $R$ be a ring. If $M$ and $N$ are finite $R$-modules such that either ${\rm pd}_RM<\infty$ or ${\rm pd}_RN<\infty$, then there exists a first quadrant spectral sequence
$$E_2^{p,q}={\rm Tor}_p^R({\rm Ext}^q_R(M,R),N)\Rightarrow_p{\rm Ext}^{q-p}_R(M,N).$$
\end{lemma}
\begin{proof}
Since $M$ or $N$ has finite projective dimension, the first quadrant double complex
$$\xymatrix@=1em{
& \vdots & \vdots & \cdots & \vdots
\\
0 & F_n^*\otimes_R G_0\ar[l]\ar[u] & F_n^*\otimes_R G_1\ar[l]\ar[u] & \cdots\ar[l] & F_n^*\otimes_R G_m\ar[l]\ar[u] & \cdots\ar[l]
\\
& \vdots & \vdots & \ddots & \vdots
\\
0 & F_1^*\otimes_R G_0\ar[l]\ar[u] & F_1^*\otimes_R G_1\ar[l]\ar[u] & \cdots\ar[l] & F_1^*\otimes_R G_m\ar[l]\ar[u] & \cdots\ar[l]
\\
0 & F_0^*\otimes_R G_0\ar[l]\ar[u] & F_0^*\otimes_R G_1\ar[l]\ar[u] & \cdots\ar[l] & F_0^*\otimes_R G_m\ar[l]\ar[u] & \cdots\ar[l]
\\
& 0\ar[u] & 0\ar[u] & \dots & 0\ar[u] & \cdots
}$$
induces two spectral sequences, where $F_\bullet$ and $G_\bullet$ are, respectively, projective resolutions of $M$ and $N$. Now the result follows from the fact that ${\rm Ext}_R^\bullet(M,N)$ can be computed through the homologies of the complex ${\rm Hom}_R(F_\bullet,R)\otimes_RN$.
\end{proof}

\medskip

Given two finite $R$-modules $M$ and $N$, we set
$$e_R(M,N):=\sup\{j\geq 0 \mid {\rm Ext}^j_R(M,N)\neq0\}.$$
The result below contributes, in particular, to \cite[Question 3.7]{KOT}. Another application of it will be given later in Corollary \ref{appl1}.

\begin{proposition}\label{attempt2}
Let $R$ be a local ring and $M$ a finite $R$-module 
satisfying $e_R({\rm Tr} M,R)=0$.
 If there exists a finite $R$-module $N$ with ${\rm depth}_RN=0$ such that ${\rm pd}_R{\rm Hom}_R(M,N)<\infty$, then $M$ is free.
\end{proposition}

\begin{proof} Pick the spectral sequence of Lemma \ref{pdss} for ${\rm Tr} M$ and ${\rm Hom}_R(M,N)$:
$$E_2^{p,q}={\rm Tor}_p^R({\rm Ext}^q_R({\rm Tr} M,R),{\rm Hom}_R(M,N))\Rightarrow_p{\rm Ext}^{q-p}_R({\rm Tr} M,{\rm Hom}_R(M,N)).$$ From the vanishing hypothesis, we obtain $E_2^{p,q}=0$ for all $q>0.$ Therefore, by convergence,
$$e_R({\rm Tr} M,{\rm Hom}_R(M,N))=0.$$ Now we apply Lemma \ref{trlemma}. 
\end{proof}

\subsection{Another criterion}
First we invoke a useful standard fact.

\begin{lemma}{\rm (\cite[Lemma 2.2]{Yo})}\label{tor0}
Let $R$ be a local ring and $M, N$ be finite $R$-modules with ${\rm pd}_RM<\infty$ and $N$ maximal Cohen-Macaulay. Then, $M$ and $N$ are {\rm Tor}-independent.
\end{lemma}

Our freeness criterion in this part is as follows. Applications will be provided in Subsection \ref{Vascon}.

\begin{corollary}\label{lmweak00}
Let $R$ be a Cohen-Macaulay local ring possessing a canonical module, and $M, N$ be finite $R$-modules with $N$ maximal Cohen-Macaulay. Suppose
$${\rm depth}\, \Omega_1^RM\otimes_RN^{\vee}\neq {\rm depth}\,M\otimes_RN^{\vee}+1$$ $($this holds, e.g., if $M\otimes_RN^{\vee}$ is maximal Cohen-Macaulay$)$. If ${\rm pd}_RM<\infty$, then $M$ is free.
\end{corollary}
\begin{proof} Notice that $N^{\vee}$ is also maximal Cohen-Macaulay and recall ${\rm pd}_RM<\infty$. In this setting, we claim that the $R$-module $M\otimes_RN^{\vee}$ is, necessarily, maximal Cohen-Macaulay. Indeed, suppose by way of contradiction that ${\rm depth}\,M\otimes_RN^{\vee}<d={\rm dim}\,R$. First, we can apply Lemma \ref{tor0} to obtain that $M$ and $N^{\vee}$ are Tor-independent. In particular, ${\rm Tor}_1^R(M, N^{\vee})=0$. Therefore, tensoring with $N^{\vee}$ a minimal free presentation $\xymatrix@=1em{0\ar[r] & \Omega_1^R M\ar[r] & F_0\ar[r] & M\ar[r] & 0}$ yields a short exact sequence $$\xymatrix{0\ar[r] &  \Omega_1^R M\otimes_RN^\vee\ar[r] & F_0\otimes_R N^\vee\ar[r] & M\otimes_RN^\vee\ar[r] & 0.}$$ Since ${\rm depth}\,M\otimes_RN^{\vee}<d={\rm depth}\,F_0\otimes_RN^{\vee}$, the well-known depth lemma forces the depth of $\Omega_1^RM\otimes_RN^{\vee}$ to coincide with ${\rm depth}\,M\otimes_RN^{\vee}+1$, contradicting our hypothesis. Thus, $M\otimes_RN^{\vee}$ is maximal Cohen-Macaulay. Now, Corollary \ref{lmweak0}(ii) gives $${\rm Ext}^j_R(M,N)=0 \quad \mbox{for\, all} \quad j\geq 1,$$ which, by \cite[p.\,154, Lemma 1(iii)]{Matsu} along with the hypothesis ${\rm pd}_RM<\infty$, forces $M$ to be free.
\end{proof}

\subsection{Generalized local cohomology}\label{dim1section} This part is the core of the section. The crucial tool here is the following notion.

\begin{definition}$($\cite{H}, \cite{S}$)$ \rm
 Let $R$ be a ring, $I$ be an ideal of $R$, and $M, N$ be finite $R$-modules. For an integer $i\geq 0$, the {\it $i$-th local cohomology module} of the pair $M, N$ with respect to $I$ is defined as
$$H_{I}^i(M,N)=\varinjlim_n{\rm Ext}^i_R(M/{I}^nM,N).$$ Notice that by letting $M=R$ we retrieve the ordinary local cohomology module $H_I^i(N)$.
\end{definition}

\begin{lemma}{\rm (\cite[Proposition 2.1]{FJMS})}\label{ss2}
 For $R$-modules $M, N$, with $M$ finite, there are spectral sequences:
\begin{itemize}
    \item [(i)]  ${\rm Ext}^p_R(M,H^q_{I}(N))\Rightarrow_p H^{p+q}_{I}(M,N)$; 
    \item [(ii)]  $H^p_{I}({\rm Ext}_R^q(M,N))\Rightarrow_p H^{p+q}_{I}(M,N)$.
    \end{itemize}
\end{lemma}

\begin{definition}\rm
The {\it cohomological dimension} of a pair of $R$-modules $M, N$, with respect to a proper ideal $I$ of $R$, is defined as
$$\cd_I(M,N)=\sup\{j\geq 0 \mid H^j_{I}(M,N)\neq 0\}.$$ Clearly, taking $M=R$, we retrieve the ordinary cohomological dimension $\cd_IN$. Also notice that this generalizes $e_R(M, N)$ since this quantity can be expressed as $\cd_{(0)}(M,N)$.
\end{definition}

Next, given a local ring $(R, {\bf m})$ and a finite $R$-module $M$, the complete intersection dimension of $M$ is denoted $\cidim_RM$. Recall that $$\cidim_RM\leq {\rm pd}_RM,$$ with equality if ${\rm pd}_RM<\infty$. Moreover, each finite $R$-module $M$ satisfies $\cidim_RM<\infty$ if and only if $R$ is a complete intersection ring (i.e., the completion of $R$ in the ${\bf m}$-adic topology is isomorphic to the quotient of a regular local ring by an ideal generated by a regular sequence), and in this case $\cidim_RM={\rm depth}\,R-{\rm depth}_RM$. For details, we refer to \cite{AGP}. 

\begin{lemma}$($\cite[Theorem 4.2]{AB}$)$\label{cidim}
 Let $R$ be a local ring and $M$ be a finite $R$-module with $\cidim_RM<\infty$. Then, ${\rm pd}_RM<\infty$ if and only if ${\rm Ext}^{2s}_R(M,M)=0$ for some  integer $s\geq 1$.
\end{lemma}

\begin{theorem}\label{freeness}
Let $R$ be a ring, $I$ be a proper ideal of $R$, and $M, N$ be $R$-modules with $M$ finite and  ${\rm dim}\,N\leq 1$. Consider the following assertions:
\begin{itemize}
    \item [(i)] $N$ is free;
    \item [(ii)] $e_R(M,H^0_{I}(N))<\infty$ and $e_R(M,H^1_{I}(N))<\infty$;
    \item [(iii)] $\cd_{I}(M, N)<\infty$;
    \item [(iv)] ${\rm Ext}^{j}_R(M,H^1_{I}(N))\cong{\rm Ext}^{j+2}_R(M,H^0_{I}(N))$ for all $j\gg 0$.
\end{itemize} Then, ${\rm (i)}\Rightarrow{\rm (ii)}\Rightarrow{\rm (iii)}\Rightarrow{\rm (iv)}$. Moreover, if in addition $R$ is local, $M=N$, ${\rm Ext}^1_R(M,M)=0$ and $\cidim_RM\leq 1$, then the first three assertions {\rm (i), (ii)} and {\rm (iii)} are equivalent.
\end{theorem}
\begin{proof}
Clearly (i) $\Rightarrow$ (ii). Since $M$ is finite, we consider the spectral sequence given in Lemma \ref{ss2}(i),
$$E_2^{p,q}={\rm Ext}^p_R(M,H^q_{I}(N))\Rightarrow_p H^{p+q}_{I}(M, N).$$ Grothendieck's vanishing theorem guarantees that $$H^q_{I}(N)=0 \quad \mbox{for\, all} \quad q\geq 2$$ (even when $N$ is not finite; see \cite[Theorem 6.1.2]{Brod}), so that $E$ has only two (possibly) non-zero rows
$$\xymatrix{
0 & 0 & 0 & 0 & \cdots
\\
E_2^{0,1}\ar[rrd] & E_2^{1,1}\ar[rrd] & E_2^{2,1} & E_2^{3,1} & \cdots
\\
E_2^{0,0} & E_2^{1,0} & E_2^{2,0} & E_2^{3,0} & \cdots
}$$ Therefore, there exist exact sequences 
$$\xymatrix{
0\ar[r] & E_\infty^{p,1}\ar[r] & E_2^{p,1}\ar[r] & E_2^{p+2,0}\ar[r] & E_\infty^{p+2,0}\ar[r] & 0,}$$
$$\xymatrix{
0\ar[r] & E_\infty^{p+2,0}\ar[r] & H^{p+2}\ar[r] & E_\infty^{p+1,1}\ar[r] & 0,}$$ where $H^i=H^i_{I}(M, N)$ for all $i\geq0$. Hence, by splicing these exact sequences, we derive a diagram
$$\xymatrix@=1em{
&&0\ar[d]
\\
\cdots\ar[r] & E_2^{p+1,0}\ar[r]\ar@{-->}[rd] & E_\infty^{p+1,0}\ar[r]\ar[d] & 0
\\
&& H^{p+1}\ar[d]\ar@{-->}[rd] \\
&0\ar[r]& E_\infty^{p,1}\ar[d]\ar[r] & E_2^{p,1}\ar[r] & E_2^{p+2,0}\ar[r]\ar@{-->}[rd] & E_\infty^{p+2,0}\ar[r]\ar[d] & 0 \\ &&0&&& H^{p+2}\ar[d]\ar@{-->}[rd]
\\ &&&&&\vdots&}$$ so that the broken arrows yield a long exact sequence 
$$\xymatrix@=1em{\cdots\ar[r] & {\rm Ext}^{p+1}_R(M,H^0_{I}(N))\ar[r] & H^{p+1}_{I}(M, N)\ar[r] & {\rm Ext}^{p}_R(M,H^1_{I}(N))\ar[dll] & \\& {\rm Ext}^{p+2}_R(M,H^0_{I}(N))\ar[r] & H^{p+2}_{I}(M, N)\ar[r] & \cdots}$$
which, in turn, enables us to conclude that (ii) $\Rightarrow$ (iii) $\Rightarrow$ (iv).

Finally, assume that (iii) holds, $R$ is local, $M=N$, ${\rm Ext}^1_R(M,M)=0$ and $\cidim_RM\leq 1$. Let us prove that $M$ is free. Consider the spectral sequence
$$E_2^{p,q}=H^p_{I}({\rm Ext}^q_R(M,M))\Rightarrow_p H^{p+q}_{I}(M,M)$$
given in Lemma \ref{ss2}(ii). As ${\rm dim}\,M\leq 1$, we have $$E_2^{p,q}=0 \quad \mbox{for\, all} \quad p\geq 2,$$ so that $E_2^{p,q}=E_\infty^{p,q}$ and $E$ has only two columns. By convergence, there is a short exact sequence
$$\xymatrix@=1em{0\ar[r] & H^1_{I}({\rm Ext}^{j-1}_R(M,M))\ar[r] & H^j_{I}(M,M)\ar[r] & H^0_{I}({\rm Ext}^j_R(M,M))\ar[r] & 0}$$ for each $j\geq 1$.
Now, by hypothesis, $H^j_{I}(M,M)=0$ for all $j\gg0$, which gives
$$H^0_{I}({\rm Ext}^i_R(M,M))=H^1_{I}({\rm Ext}^i_R(M,M))=0 \quad \mbox{for\, all} \quad i\gg 0.$$ As ${\rm dim}\,{\rm Ext}^i_R(M,M)\leq {\rm dim}\,M\leq 1$, this forces $${\rm Ext}^i_R(M,M)=0 \quad \mbox{for\, all} \quad i\gg 0,$$ and then Lemma \ref{cidim} ensures that ${\rm pd}_RM<\infty$. Hence ${\rm pd}_RM=\cidim_RM\leq 1$, which, by a well-known device (see \cite[p.\,154, Lemma 1(iii)]{Matsu}), yields 
$$e_R(M,M)={\rm pd}_RM\leq 1.$$ Because ${\rm Ext}^1_R(M,M)=0$, we conclude  $e_R(M,M)=0$, so that $M$ is free.
\end{proof}

\begin{remark}\rm Clearly, if in addition $N$ is Cohen-Macaulay with respect to $I$ (meaning $H_I^i(N)=0$ for all $i\neq {\rm cd}_IN$, which is a generalization of the usual Cohen-Macaulay property over local rings), then assertion (ii) can be replaced with $$e_R(M,H^{{\rm cd}_IN}_{I}(N))<\infty.$$ \end{remark}

\begin{remark}\rm For finite $R$-modules $M, N$, we point out that if $e_R(M, N)<\infty$ then ${\rm cd}_I(M, N)< \infty$ for any proper ideal $I$ of $R$. This is an immediate consequence of Lemma \ref{ss2}(ii). Now it seems natural to ask whether condition (ii) of Theorem \ref{freeness} implies $e_R(M, N)<\infty$ (this is trivial if $I=(0)$).
    
\end{remark}

\begin{corollary}\label{CM-tensor-Tor}
Let $R$ be a Cohen-Macaulay local ring possessing a canonical module, and let $M$ be a maximal Cohen-Macaulay $R$-module such that $M\otimes_R M^\vee$ is also maximal Cohen-Macaulay and ${\rm Tor}^R_1(M,M^\vee)=0$. If $\cidim_RM\leq 1$ and ${\rm cd}_I(M, M)< \infty$ for some proper ideal $I$ of $R$, then $M$ is free.
\end{corollary}
\begin{proof} By Corollary \ref{lmweak0}(i) with $n=1$, we obtain 
$${\rm Ext}^1_R(M,M)=0.$$ Now, we apply Theorem \ref{freeness}.
\end{proof}

\medskip

Further byproducts of Theorem \ref{freeness}, or more precisely versions of Corollary \ref{CM-tensor-Tor} when $R$ is 1-dimensional and Gorenstein or a complete intersection ring, will be given in Subsection \ref{HWc} concerning the well-studied Huneke-Wiegand conjecture. 

\section{Applications}

\subsection{Codimension 3 complete intersections} We start this last section providing  a curious characterization of complete intersection ideals of height 3. Recall that an ideal $J$ of a ring $S$ is  {\it generically a complete intersection} if, for each associated prime ideal ${\bf p}$ of the $S$-module $S/J$, the ideal $J_{{\bf p}}$ is generated by an $S_{{\bf p}}$-sequence. For instance, any radical ideal in a regular local ring is generically a complete intersection.

\begin{corollary}\label{appl1}
Let $(S, {\bf n})$ be a regular local ring and $J$ be generically a complete intersection ideal of $S$ with ${\rm height}\,J= 3$, such that the ring $R=S/J$ is Gorenstein. If there exists a finite $R$-module $N$ satisfying $${\rm depth}_RN=0 \quad \mbox{and} \quad {\rm pd}_R{\rm Hom}_S(J, N)<\infty,$$ then $J=(f, g, h)$ for some $S$-sequence $\{f, g, h\}\subset {\bf n}$, and conversely.
\end{corollary}
\begin{proof} Since $R$ is Gorenstein and $J$ has codimension 3, the well-known Buchsbaum-Eisenbud structure theorem (see, e.g., \cite[Theorem 3.4.1]{BH}) implies that the conormal module $J/J^2$ is the cokernel of an alternating map of free $R$-modules, which gives $${\rm Tr}\,J/J^2\cong J/J^2.$$ Moreover, in this setting (using in particular that $J$ is generically a complete intersection), the $R$-module $J/J^2$ is known to be maximal Cohen-Macaulay (see \cite{H2}, \cite{HU}). It follows that
$$e_R({\rm Tr}\,J/J^2, R)=0.$$ Furthermore, by adjunction we have $R$-module isomorphisms 
$${\rm Hom}_S(J, N)\cong {\rm Hom}_S(J, {\rm Hom}_R(R, N))\cong {\rm Hom}_R(J\otimes_SR, N)\cong {\rm Hom}_R(J/J^2, N).$$
Now, applying Proposition \ref{attempt2}, we obtain that $J/J^2$ is free, which by \cite{Vas} means that $J$ is generated by a regular sequence (of length necessarily 3 in this case).

Conversely, suppose $J$ is generated by an $S$-sequence.  Thus, $J/J^2$ is $R$-free and so, if we take $$N=R/({\bf x}), \quad 
\{{\bf x}\}\subset {\bf n}R \quad \mbox{a system of parameters},$$ then ${\rm depth}_RN=0$ and the $R$-module ${\rm Hom}_R(J/J^2, N)$ is isomorphic to a direct sum of -- three, in the present case -- copies of $R/({\bf x})$, which must have finite projective dimension since $\{\bf x\}$ is an $R$-sequence (as $R$ is Cohen-Macaulay).
\end{proof}

\subsection{The Herzog-Vasconcelos conjecture and the multi-conjecture of Vasconcelos}\label{Vascon} We now relate Corollary 
\ref{lmweak00} to some long-standing conjectures. First, let $R$ stand for an algebra essentially of finite type over a field $k$ of characteristic zero. We denote by ${\rm D}_{R/k}$ the module of $k$-derivations of $R$. The so-called Herzog-Vasconcelos conjecture predicts that if $${\rm pd}_R{\rm D}_{R/k}<\infty$$ then ${\rm D}_{R/k}$ is free. A major case where this problem has been settled affirmatively is when $R$ is a quasi-homogeneous isolated complete intersection singularity. We refer to \cite{H3}.

\begin{corollary}\label{appl3}
The Herzog-Vasconcelos conjecture holds true if $R$ is Cohen-Macaulay and there exists a  maximal Cohen-Macaulay $R$-module $N$ such that $${\rm depth}\, \Omega_1^R{\rm D}_{R/k}\otimes_RN^{\vee}\neq {\rm depth}\,{\rm D}_{R/k}\otimes_RN^{\vee}+1$$ 
$($e.g., if  ${\rm D}_{R/k}\otimes_RN^{\vee}$ is maximal Cohen-Macaulay$)$.
\end{corollary}
\begin{proof} This follows immediately from Corollary 
\ref{lmweak00}.
\end{proof}

\begin{corollary}\label{appl3bis}
The Herzog-Vasconcelos conjecture holds true if $R$ is Cohen-Macaulay and $${\rm depth}\, \Omega_1^R{\rm D}_{R/k}\otimes_R\omega_R\neq {\rm depth}\,{\rm D}_{R/k}\otimes_R\omega_R+1.$$ 
\end{corollary}
\begin{proof} Apply Corollary 
\ref{appl3} with $N=R$.
\end{proof}

\begin{remark}\rm At this point it is worth noting that, in general, the property of a tensor product $T\otimes_RU$ (of finite $R$-modules $T, U$) being maximal Cohen-Macaulay over a local ring $R$ does {\it not} force either $T$ or $U$ to have the same property; see, e.g., \cite[Example 2.1]{Consta}. \end{remark}

Next, and following the same spirit, we recall part of a (still open) multi-conjecture due to Vasconcelos (see \cite[p.\,373]{V}). Let $R=S/J$, where $S$ is a regular local ring containing a field $k$ of characteristic zero and $J$ is an ideal of $S$, and let $E$ stand for either $$(J/J^2)^* \quad \mbox{or} \quad J/J^2\otimes_R\omega_R,$$  which are, respectively, the normal module and the twisted conormal module  of $J\subset S$. In case $S$ is a localization of a polynomial ring over $k$, the module $E$ is also allowed to be the module ${\rm Diff}_{k}(R)$ of K\"ahler $k$-differentials of $R$ (whose $R$-dual is just the derivation module ${\rm D}_{R/k}$). Then, Vasconcelos' conjecture states that $J$ must be generated by a regular sequence whenever $${\rm pd}_RE<\infty.$$

\begin{corollary}\label{appl4} Assume that $R$ is Cohen-Macaulay. In case $E=(J/J^2)^*$, suppose $J/J^2$ is reflexive. Then, for $E$ as in the list above, Vasconcelos' conjecture holds true if there exists a  maximal Cohen-Macaulay $R$-module $N$ such that 
$${\rm depth}\, \Omega_1^RE\otimes_RN^{\vee}\neq {\rm depth}\,E\otimes_RN^{\vee}+1$$ $($this holds, e.g., if $E\otimes_RN^{\vee}$ is maximal Cohen-Macaulay$)$.
\end{corollary}
\begin{proof} First, assume $E=(J/J^2)^*$. By Corollary 
\ref{lmweak00}, the $R$-module $(J/J^2)^*$ is free. Since $J/J^2$ is reflexive by assumption, we obtain that it must be free and hence, by \cite{Vas}, $J$ is generated by an $S$-sequence. Next, let us consider the case $E=J/J^2\otimes_R\omega_R$. By Corollary 
\ref{lmweak00}, this module must be free. Now recall the basic fact that, for non-zero finite modules $T, U$ over a local ring $R$, if the tensor product $T\otimes_RU$ is free, then $T$ and $U$ are also free  (we supply here a proof for the reader's convenience: the condition of $T$ and $U$ being non-zero is equivalent to $T\otimes_RU\neq 0$. Now pick an $R$-free presentation $$\xymatrix{F_1\ar[r] & F_0\ar[r] & T\ar[r] & 0}$$ of the module $T$. Tensoring it with $U$, and using that the module $G=T\otimes_RU$ is free, the resulting map splits so that $G$ is a direct summand of $F_0\otimes_RU$. Tensoring such a decomposition with $T$, we obtain that $G\otimes_RT$ is a direct summand of the module $$F_0\otimes_RU\otimes_RT=F_0\otimes_RG,$$ which is free, i.e., $G\otimes_RT$ is projective, hence free because $R$ is local. This is equivalent to $T$ being free. Clearly, a completely analogous argument shows $U$ must be free as well). It now follows that $J/J^2$ is free, which as we know forces $J$ to be generated by an $S$-sequence. Finally, if $E={\rm Diff}_{k}(R)$, where $R$ is essentially of finite type over $k$, then using again  Corollary \ref{lmweak00} we obtain that ${\rm Diff}_{k}(R)$
is free, which is equivalent to $R$ being regular (see \cite[Theorem 14.1]{K}) and hence in particular $J$ can be generated by an $S$-sequence. \end{proof}

\begin{corollary}\label{appl4bis} Assume that $R$ is Cohen-Macaulay. In case $E=(J/J^2)^*$, suppose $J/J^2$ is reflexive. Then, for $E$ as in the list above, Vasconcelos' conjecture holds true if 
$${\rm depth}\, \Omega_1^RE\otimes_R\omega_R\neq {\rm depth}\,E\otimes_R\omega_R+1.$$ 
\end{corollary}
\begin{proof} Apply Corollary \ref{appl4} with $N=R$.
\end{proof}

\begin{remark} \rm The situation where $E=J/J^2$, the conormal module, also appears in the original statement of Vasconcelos' conjecture (in fact it is the first module in the list) but this case has already been settled positively in \cite{B}, where moreover the case $E={\rm Diff}_{k}(R)$ has been confirmed under the additional assumption that $R$ admits no non-trivial evolution. \end{remark}

\begin{remark} \rm If $S$ is either a localized polynomial ring $k[x_1, \ldots, x_n]_{(x_1, \ldots, x_n)}$ or a formal power series ring $k[\!]x_1, \ldots, x_n]\!]$, and if $J$ is radical, then it is well-known that $$\Omega_1^R{\rm Diff}_{k}(R)=J/J^{(2)},$$ where $J^{(2)}=\bigcap_{{\bf p}\in {\rm Ass}_SR}J^2S_{\bf p}\cap S$, the second symbolic power of $J$. \end{remark}

\begin{question} \rm Let $E$ be as in Vasconcelos' list. What can be said about $R$ (concerning, e.g., the properties of $R$ being Gorenstein, or a complete intersection, or normal, or regular) if $$\cidim_RE<\infty?$$ What if we replace $\cidim_RE$ with the Gorenstein dimension $\gdim_RE$ of $E$? \end{question} 

\subsection{The Huneke-Wiegand conjecture}\label{HWc} Recall that a finite $R$-module $M$ is said to have a rank if $M\otimes_R Q$ is a free $Q$-module, where $Q$ is the total quotient ring of $R$. In this case, the rank of $M$ is the integer $r\geq 0$ such that $M\otimes_R Q\cong Q^r$. The celebrated Huneke-Wiegand conjecture (see \cite{HW}, also \cite{HIW}) states that, if $R$ is a one-dimensional Gorenstein local ring
and $M$ is a torsionfree finite $R$-module having a rank, and satisfying the property that $M \otimes_R M^*$ is torsionfree, then $M$ must be free. In fact the original statement of the conjecture requires $R$ to be a domain, but even in the present more general setting no counterexample has been produced so far. However, it is  worth highlighting that the hypothesis of $M$ having a rank cannot be dropped, as the following standard example shows: $$R={\mathbb C}[\![x, y]\!]/(xy) \quad \mbox{and} \quad M=R/xR\cong yR.$$ 

\begin{lemma}$($\cite[Proposition 4.3]{HIW}$)$\label{rigid}
 Let $R$ be a one-dimensional Gorenstein local ring and let $M$ be a  torsionfree finite $R$-module with positive rank. Then, $M\otimes_R M^*$ is torsionfree if and only if ${\rm Ext}^1_R(M,M)=0$.
\end{lemma}

\begin{corollary}\label{Goren} The Huneke-Wiegand conjecture holds true if $\cidim_RM<\infty$ and  ${\rm cd}_I(M, M)< \infty$ for some proper ideal $I$ of $R$.
\end{corollary}
\begin{proof} First, in case the rank of $M$ is zero, we have $M\otimes_RQ=0$. On the other hand, since $M$ is torsionfree, the natural map 
$\xymatrix@=1em{M\ar[r] & M\otimes_RQ}$ is injective. Hence, $M=0$ and there is nothing to prove. So we may assume that the rank of $M$ is positive. Now Lemma \ref{rigid} gives ${\rm Ext}^1_R(M,M)=0$. Also, $\cidim_RM\leq {\rm depth}\,R=1$. Finally, we apply Theorem \ref{freeness}.
\end{proof}

\begin{corollary}\label{H-W-version}
The Huneke-Wiegand conjecture in case $R$ is a complete intersection holds true if $${\rm cd}_I(M, M)< \infty$$ for some proper ideal $I$ of $R$.
\end{corollary}

Since the conjecture remains open in the complete intersection case (of codimension at least $2$), this corollary has the 
particular effect of highlighting the relevance of ${\rm cd}_I(M, M)$.

\end{document}